\documentclass[12pt, reqno]{amsart}
\usepackage{amsmath, amsthm, amscd, amsfonts, amssymb, graphicx, color}

\newtheorem{theorem}{Theorem}[section]
\newtheorem{lemma}[theorem]{Lemma}

\newtheorem{corollary}[theorem]{Corollary}
\theoremstyle{definition}
\newtheorem{definition}[theorem]{Definition}
\newtheorem{example}[theorem]{Example}

\newtheorem{question}[theorem]{Question}

\theoremstyle{remark}
\newtheorem{remark}[theorem]{Remark}
\begin{document}

\setcounter{page}{1}

\title{On $e^*$-$\theta$-$D$-sets and Related Topics}

\author{D. AKALIN$^{\rm a}$, M. Özkoç$^{\rm b,\ast}$}

\address{$^{A}$Muğla Sıtkı Koçman University, Graduate School of Natural and Applied Sciences, Mathematics, 48000, Menteşe-Muğla, Turkey.}

\email{dilekkakaln23@gmail.com}

\address{$^{B}$Muğla Sıtkı Koçman University, Faculty of Science, Department of Mathematics, 48000, Menteşe-Muğla, Turkey.}

\email{murad.ozkoc@mu.edu.tr}




\date{--- $^{*}$Corresponding author}

\begin{abstract}

This paper aims to study the notion of $e^*\text{-}\theta$-open \cite{Ozkoc1} sets and to investigate new properties of this notion. Also, we define a new type of set, called $e^*\text{-}\theta\text{-}D$-set, via the notion of $e^*\text{-}\theta$-open set. Moreover, we introduce some new separation axioms by utilizing $e^*\text{-}\theta\text{-}D$-sets. We obtain many results related to these new notions. In addition, the notions of $e^*\text{-}\theta\text{-}$kernel and slightly $e^*\text{-}\theta\text{-}R_0$ space are defined. Some characterizations regarding these new notions have been obtained. Furthermore, we have given many examples concerning the mentioned notions. Finally, we not only put forward the definition of $e^*\text{-}R_1$ space but also obtained some of its characterizations.

\end{abstract}
\maketitle

\section[Introduction]{Introduction}


The concept of open set is very important for general topology. The broad perspective offered to mathematicians by this fundamental concept in topology facilitates the organization of elements and the study of topological spaces. This richness of the open set concept adds depth to mathematical thought and contributes significantly to the understanding of topological spaces. 
As we go deeper into topology, the diversity of open sets increases. This diversity provides a mathematical basis for separating different points and set structures with different neighborhoods. 
At this stage, it is important to point out that separation axioms play a decisive role. Today, various special and general forms of separation axioms have been defined and studied in detail by many mathematicians. Moreover, special and general forms of open sets have also gained an important place in the mathematical literature. 
 For instance, in 1966, Velicko introduced the notion of $\theta$-open \cite{Velicko} set which is the stronger form of open sets in topology. After him, several new forms of $\theta$-open \cite{Velicko} classes such as pre-$\theta$-open \cite{Noiri}, semi-$\theta$-open \cite{Jafari}, $\beta$-$\theta$-open \cite{Caldas,Popa}, and $e^*\text{-}\theta$-open \cite{Ozkoc1} were defined and studied in the literature. In the first part of this paper, we present an independent notion of $\beta\text{-}\theta\text{-}$open \cite{Caldas,Popa} set, which was introduced in 2015 by Farhan et. al.
The literature on $e^*\text{-}\theta$-open and $e^*\text{-}\theta$-closed sets defined by Yang
some new results are obtained that do not hold. In the second section, the $e^* \text{-}\theta\text{-}D$-set and quasi $e^* \text{-}\theta\text{-}D$-set defined in this paper are analyzed.
The concepts of $e^*\text{-}\theta \text{-}D_0$ space, $e^*\text{-}\theta \text{-}D_1$ space, $e^*\text{-}\theta \text{-}D_2$ space are introduced through the concepts of $\theta$-closed set and
the relations between these concepts are investigated.
In the third section, the concepts of $e^*\text{-}\theta\text{-}R_{0}$ space and $e^*\text{-}\theta$-core are partially defined and the relations between these two concepts are analyzed.
In the fourth section, the notions of $S$-continuous, $\theta\text{-}S\text{-}e^*$-continuous and $S\text{-}e^*$-continuous functions are defined and some basic results on these functions are presented. In the last section, the definition of $e^*\text{-}R_{1}$ space is given and two characterizations of this concept has been obtained.\\

\section[Preliminaries]{Preliminaries}
Throughout this paper, $X$ and $Y$ refer always to topological spaces on which no
separation axioms are assumed unless otherwise mentioned. For a subset $A$ of $X,$
$cl(A)$ and $int(A)$ denote the closure of $A$ and the interior of $A$ in $X,$ respectively. The family of
all open subsets containing $x$ of $X$ is denoted by $O(X,x).$ A subset $A$ is said to be regular open \cite{Stone} (resp. regular closed \cite{Stone}) if $A = int(cl(A))$ (resp. $A = cl(int(A))$). The $\delta \text{-}$interior \cite{Velicko} of a subset $A$ of $X$ is the union of all regular open sets of $X$ contained in $A$ and is denoted by $\delta\text{-}int(A).$ The subset $A$ of a space $X$ is called $\delta \text{-}$open \cite{Velicko} if $A =\delta\text{-}int(A),$ i.e., a set is $\delta\text{-}$open if it is the union of some regular open sets. The complement of a $\delta\text{-}$open set is called $\delta \text{-}$closed. Alternatively, a subset $A$ of a space $X$ is called $\delta$-closed \cite{Velicko} if $A =\delta\text{-}cl(A),$ where $\delta\text{-}cl(A) = \{x\in X|(\forall U\in O(X,x))(int(cl(U))\cap A \neq \emptyset)\}.$ The family of all $\delta\text{-}$open (resp. $\delta\text{-}$closed) sets in $X$ is denoted by $\delta O(X)$ (resp. $\delta C(X)$).
\\

$A$ subset $A$ of $X$ is said to be $e^*\text{-}$open  \cite{Ekici2} if $A \subseteq cl(int(\delta \text{-}cl(A))).$ The complement
of an $e^*\text{-}$open set is called $e^*\text{-}$closed. A subset $A$ of a space $X$ is said to be $e^*$-regular  \cite{Ozkoc1} if it is both $e^*$-open and $e^*$-closed. The $e^*\text{-}$interior \cite{Ekici2} of a subset $A$ of $X$ is the union of all $e^*\text{-}$open sets of $X$ contained in $A$ and is denoted by $e^*\text{-}int(A).$ The $e^*$-closure \cite{Ekici2} of a subset $A$ of $X$ is the intersection of all $e^*$-closed sets of $X$ containing $A$ and is denoted by $e^*\text{-}cl(A).$ 
The family of all $e^*$-open  (resp. $e^*$-closed, $e^*$-regular) subsets containing $x$ of $X$ is denoted by $e^*O(X,x)$ (resp. $e^*C(X,x),$ $e^*R(X,x)).$ 
\\

A point $x$ of $X$ is called an $e^*\text{-}\theta$-cluster  \cite{Ozkoc1} point of $A \subseteq X$ if $e^*\text{-}cl(U) \cap A = \emptyset$ for
every $U \in e^*O(X, x).$ The set of all $e^*\text{-}\theta\text{-}$cluster points of $A$ is called the $e^*\text{-}\theta\text{-}$closure
of $A$ and is denoted by $e^*\text{-}cl_{\theta}(A).$ A subset $A$ is said to be $e^*\text{-}\theta\text{-}$closed if and only if $A=e^*\text{-}cl_{\theta}(A).$ The complement of an $e^*\text{-}\theta\text{-}$closed set is said to be $e^*\text{-}\theta\text{-}$open. The family of
all $e^*\text{-}\theta$-closed (resp. $e^*\text{-}\theta$-open) subsets of $X$ is denoted by $e^*\theta C(X)$ (resp. $e^*\theta O(X)).$ The family of
all $e^*\text{-}\theta$-closed (resp. $e^*\text{-}\theta$-open) subsets containing $x$ of $X$ is denoted by $e^*\theta C(X,x)$ (resp. $e^*\theta O(X,x)).$ Also, the family of all $e^*\text{-}\theta $-open sets containing the subset $F$ of $X$ will be denoted by $e^*\theta O(X,F).$
\\
\begin{theorem} $($\cite{Ozkoc1}$)$ The following properties hold for a subset $A$ of a topological space $(X , \tau):$ \\
$(a)$ $A\in e^*O(X)$ \text{if and only if }$ e^*\text{-}cl(A)\in e^*R(X),$\\
$(b)$ $A\in e^*C(X)$ \text{if and only if }$ e^*\text{-}int(A)\in e^*R(X).$
\end{theorem}

\begin{corollary}\label{5}     $($\cite{Ozkoc1}$)$
	Let $A$ and $A_{\alpha}(\alpha \in \Lambda)$ be any subsets of a space $(X, \tau ).$ Then the following properties hold:\\
	$(a)$ $A$ is $e^*\text{-}\theta\text{-}$open in $X$ if and only if for each $x \in A$ there exists $U \in e^*R(X,x)$ such that $x \in U \subseteq A,$\\
	$(b)$ If $A_{\alpha}$ is  $e^*\text{-} \theta \text{-}$open in $X$ for each $\alpha \in A,$ then $\cup_{\alpha \in \Lambda}A_{\alpha}$ is  $e^*\text{-} \theta \text{-}$open in $X.$
	\end{corollary}

	\begin{theorem} \label{1}
$($\cite{Ozkoc1}$)$ For a subset $A$ of a topological space $X ,$ the following
	properties hold:\\
	$(a)$ If $A\in e^*O(X)$, then $e^*\text{-}cl(A)=e^*\text{-}cl_\theta(A),$\\
	$(b)$ $A\in e^*R(X)$ if and only if $A$ is $e^*\text{-}\theta\text{-}open$
	and $e^*\text{-}\theta\text{-}closed.$
	\end{theorem}
	
	\begin{definition}
	$($\cite{Ozkoc1}$)$
	 $A$ space $X$ is said to be $e^*$-regular if for each closed set $F \subseteq X$ and each point $x\in X \setminus F,$ there exist disjoint $e^*$-open sets $U$ and $V$ such that $x \in U$ and $F \subseteq V.$
	\end{definition}
	\begin{theorem} $($\cite{Ozkoc1}$)$
    For a space $X,$ the following are equivalent: \\
    $(a)$ $X$ is $e^*\text{-}regular;$ \\
    $(b)$ For each point $x\in X$ and for each open set $U$ of $X$ containing $x,$ there exists $V\in e^*O(X)$ such that $x \in V \subseteq e^*\text{-}cl(V) \subseteq U;$ \\
    $(c)$ For each subset $U\in e^*O(X)$ and each  $x\in U,$ there exists $V\in e^*R(X)$ such that $x \in V \subseteq U.$
	\end{theorem}
	\begin{definition} $($\cite{Ekici1}$)$
	  $A$ function $f: X \to Y$ is said to be $e^*\text{-}$irresolute if $f^{-1}[V]\in e^*O(X)$ for every $V\in e^*O(Y).$
	\end{definition}
	\begin{theorem} \label{3} $($\cite{ozkocatasever}$)$
    Let $f : X \rightarrow Y$ be a function. The following properties are equivalent:\\
    $(a)$ $f$ is weakly $e^*\text{-}$irresolute;\\
    $(b)$ $f[e^*\text{-}cl(A)] \subseteq e^*\text{-}cl_{\theta}(f[A])$ for every subset $A$ of $X,$\\
    $(c)$
    $f^{-1}[V]$ is $e^*\text{-}\theta \text{-} open$ in $X$ for every $e^*\text{-}\theta \text{-} open$ $V$ of $Y.$
	\end{theorem}
	\begin{remark}$($\cite{Ozkoc1}$)$
	It can be easily shown that $e^*\text{-}$regular $\Rightarrow$ $e^*\text{-}\theta$-open $\Rightarrow$ $e^*\text{-}$open. 
	\end{remark}
	\begin{theorem}\label{kapanis} $($\cite{Ozkoc1}$)$
	For each subset $A$ of a topological space $(X,\tau),$ we have: 
$$\begin{array}{rcl} e^*\text{-}cl_{\theta}(A) & =  &\bigcap\{V|A\subseteq V \text{ and } V \text{ is } e^*\text{-}\theta\text{-}closed\} \\
	 &= &\bigcap \{V|A\subseteq V \text{ and } V\in e^*R(X)\}.
	  \end{array}$$
	\end{theorem}

\begin{remark}
It is easy to prove that: \\
$(a)$ the intersection of an arbitrary collection of 	$e^*\text{-}\theta$-closed sets is $e^*\text{-}\theta$-closed. \\
$(b)$ $X$ and $\emptyset$ are $e^*\text{-}\theta$-closed sets.
\end{remark}
\begin{remark}
The following example shows that the union of any two $e^*\text{-}\theta$-closed sets of $X$ need not be $e^*\text{-}\theta$-closed in $X.$
\end{remark}

\begin{example} (\cite{Ozkoc1})
Let $X=\{1,2,3\}$ and $\tau=\{\emptyset,X,\{1\},\{2\},\{1,2\}\}.$ The subsets $\{1\}$ and $\{2\}$ are $e^*\text{-}\theta$-closed in $(X,\tau)$ but $\{1,2\}$ is not $e^*\text{-}\theta$-closed.
\end{example}

\begin{theorem}\label{185}  $($\cite{Ozkoc1}$)$ For a subset $A$ of a space $X,$ the following are equivalent: \\
$(a)$ $A\in e^*R(X);$\\
$(b)$ $A = e^*\text{-}cl(e^*\text{-}int(A));$\\
$(c)$ $A = e^*\text{-}int(e^*\text{-}cl(A)).$
\end{theorem}
\begin{theorem} \label{8}   $($\cite{Ozkoc1}$)$ Let $X$ be a topological space and $A\subseteq X.$ Then, the operator $e^*\text{-}cl$ is idempotent, i.e., 
 $e^*\text{-}cl_{\theta}(e^*\text{-}cl_{\theta}(A)) = e^*\text{-}cl_{\theta}(A).
   $ \end{theorem}

\section[On $e^*\text{-}\theta$-open sets]{On $e^*\text{-}\theta$-open sets}      
 \begin{definition}\label{tanım}
$A$ subset $A$ of a topological space $X$ is said to be $\theta \text{-}$complement $e^*\text{-}$open (briefly, $\theta \text{-}c \text{-} e^*\text{-}$open$)$ provided there exists a subset $U$ of $X$ for which $X\setminus A= e^*\text{-}cl_{\theta}(U).$ We call a set $\theta \text{-}$complement $e^*\text{-}$closed if its complement is $\theta \text{-}c \text{-} e^*\text{-}$open. 
 \end{definition} 
\begin{remark} \label{uyarı}
        It should be mentioned that by Lemma \ref{8}, $X\setminus A = e^*\text{-}cl_{\theta}(U)$ is $e^*\text{-}\theta\text{-}$closed and $A$ is $e^*\text{-}\theta\text{-}$open. Therefore, the equivalence of $\theta\text{-}c\text{-}e^*\text{-}$open and $e^*\text{-}\theta\text{-}$open is obvious from Definition \ref{tanım}.
\end{remark}
\begin{theorem}
Let $X$ be a topological space and $A\subseteq X.$ If $A$ is $e^*\text{-}$open, then $e^*\text{-}int(e^*\text{-}cl_{\theta}(A))$ is $e^*\text{-}\theta\text{-}open.$
\end{theorem}
\begin{proof}
Let $A\in e^*O(X).$\\
$\begin{array}{l}A\in e^*O(X) \Rightarrow e^*\text{-}cl(A) \in e^*C(X) \Rightarrow X \setminus e^*\text{-}cl(A) \in e^*O(X)\\
\overset{\text{Theorem }\ref{185} }{\Rightarrow}  e^*\text{-}cl(X\setminus e^*\text{-}cl(A))= e^*\text{-}cl_{\theta}(X \setminus e^*\text{-}cl(A)) \ldots(1)
 \end{array}$ 
\\
$
\left.\begin{array} {rr}
X \setminus e^*\text{-}int(e^*\text{-}cl(A)) = e^*\text{-}cl(X\setminus e^*\text{-}cl(A)) \Rightarrow e^*\text{-}int(e^*\text{-}cl(A))=\setminus e^*\text{-}cl(\setminus e^*\text{-}cl(A)) \\
A \in e^*O(X)
\end{array}\right\} \Rightarrow$
\\
$\begin{array}{l}\Rightarrow e^*\text{-}int(e^*\text{-}cl_{\theta}(A))=e^*\text{-}int(e^*\text{-}cl(A))= X\setminus e^*\text{-}cl(X\setminus e^*\text{-}cl(A)) \ldots(2)\end{array}$ \\
$ (1),(2)$ $\Rightarrow e^*\text{-}int(e^*\text{-}cl_{\theta}(A))= \setminus e^*\text{-}cl_{\theta}(\setminus e^*\text{-}cl(A))\Rightarrow e^*\text{-}int(e^*\text{-}cl_{\theta}(A)) \in e^* \theta O(X).$
\end{proof}
\begin{theorem}
Let $X$ be a topological space. Then the notion of $e^*\text{-}\theta\text{-}$open is equivalent to the notion of $e^*\text{-}$regular if and only if $e^*\text{-}cl_{\theta}(A)$ is $e^*\text{-}$regular for every set $A\subseteq X.$ 
\end{theorem}
\begin{proof}
$(\Rightarrow):$ Let $e^*\theta O(X)=e^*R(X)$ and $A\subseteq X. $\\
$A\subseteq X \Rightarrow e^*\text{-}cl_{\theta}(A) \overset{\text{Theorem }\ref{8} } {=}e^*\text{-}cl_{\theta}(e^*\text{-}cl_{\theta}(A)) \Rightarrow e^*\text{-}cl_{\theta}(A) \in e^*\theta C(X) \ldots (1)$\\
$\left.\begin{array}{rr}
e^*\text{-}cl_{\theta}(A) \in e^*\theta C(X) \Rightarrow X\setminus e^*\text{-}cl_{\theta}(A) \in e^* \theta O(X)\\
e^*\theta O(X) =e^*R(X)
\end{array}\right\} \Rightarrow$ \\
$\begin{array}{l}\Rightarrow$ $X \setminus e^*\text{-}cl_{\theta}(A) \in e^*R(X)\subseteq  e^*\theta C(X) \end{array}$ \\
$\begin{array}{l}\Rightarrow e^*\text{-}cl_{\theta}(A) \in e^*\theta O(X) \ldots(2)\end{array}$ \\
$(1),(2)\Rightarrow e^*\text{-}cl_{\theta}(A)\in e^*R(X).$ \\ \\
$(\Leftarrow):$ Let $U\in e^*\theta O(X).$ Our aim is to show that $U\in e^*R(X).$ 
\\
$ \left.\begin{array}{rr}
U\in e^*\theta O(X) \overset{\text{Remark }\ref{uyarı} }{\Rightarrow} (\exists A\subseteq X)(X\setminus U=e^*\text{-}cl_{\theta}(A))  \\
\text{Hypothesis}
\end{array}\right\} \Rightarrow X\setminus U \in e^*R(X) $ \\
$\begin{array}{l}\Rightarrow U\in e^*R(X).\end{array}$
\end{proof}
\begin{theorem}
Let $X$ be a topological space and $B\subseteq X.$ If $B$ is $e^*\text{-}\theta$-open, then $B$ is an union some of $e^*\text{-}$regular sets.
\end{theorem}
\begin{proof}
Let $B\in e^*\theta O(X)$ and $x\in B.$\\
$\left.\begin{array}{rr}
B\in e^*\theta O(X)\overset{\text{Remark }\ref{uyarı} }{\Rightarrow} (\exists A\subseteq X)(B= X\setminus e^*\text{-}cl_{\theta}(A)) \\
x\in B
\end{array}\right\} \Rightarrow  x\notin e^*\text{-}cl_{\theta}(A) $ \\
$\begin{array}{l}\Rightarrow (\exists W_{x} \in e^*O(X,x))(e^*\text{-}cl(W_{x})\cap A= \emptyset)
\end{array}$ \\
$\begin{array}{l}\Rightarrow (\exists W_{x} \in e^*O(X,x))(e^*\text{-}cl(W_{x})\subseteq \setminus A)
\end{array}$ \\
$\left.\begin{array}{rr}
\Rightarrow (W_x \in e^*O(X,x))(e^*\text{-}int(e^*\text{-}cl(W_{x}))=(e^*\text{-}int_{\theta }(e^*\text{-}cl(W_{x}))\subseteq e^*\text{-}int_{\theta }(\setminus A)= \setminus e^*\text{-}cl_{\theta}(A) \\
\mathcal{A}:=\{e^*\text{-}int(e^*\text{-}cl(W_{x}))| (\forall x\in B)(\exists W_{x}\in e^*O(X,x))(e^*\text{-}int(e^*\text{-}cl(W_{x}))\subseteq  \setminus e^*\text{-}cl_{\theta}(A))\}
\end{array}\right\} \Rightarrow$ \\
$\begin{array}{l}\Rightarrow(\mathcal{A}\subseteq e^*R(X))(B=\bigcup \mathcal{A}).
\end{array}$
\end{proof}

\begin{corollary}
Let $X$ be a topological space and $B\subseteq X.$ If $B$ is $e^*\text{-}\theta$-closed, then $B$ is an intersection some of $e^*$-regular sets.
\end{corollary}

\section[On $e^*\text{-}\theta\text{-}D_{i}$ and $e^*\theta\text{-}T_{i}$ topological spaces]{On $e^*\text{-}\theta\text{-}D_{i}$ and $e^*\theta\text{-}T_{i}$ topological spaces}
In this chapter, we introduce some classes of sets via the notion of $e^*\text{-}\theta$-open sets. Also, the relationships between these notions and some other existing notions in the literature are investigated.

\begin{definition} \label{6}
A subset $A$ of a topological space $X$ is called an $e^*\text{-}\theta\text{-}D$-set if there exist two sets $U,V\in e^*\theta O(X)$ such that $U \neq X$ and $A=U\setminus V.$ The family of all $e^*\text{-}\theta\text{-}D$-set of $X$  and all $e^*\text{-}\theta\text{-}D$-set of $X$ containing $x\in X$ will be denoted by $e^*\theta D(X)$ and $e^*\theta D(X,x),$ respectively.
\end{definition}
\begin{remark}
It is clear that every $e^*\text{-}\theta$-open set $U$ different from $X$ is an $e^*\text{-}\theta\text{-}D\text{-}$set. However, the converse of this implication need not be true as shown by the following example.
\end{remark}

\begin{example}
Let $X=\{a,b,c,d\}$ and $\tau=\{\emptyset,X,\{a\},\{c\},\{a,c\},\{c,d\},\{a,c,d\}\}.$ Then, $e^*\theta O(X)=2^X\setminus \{\{b\}\}$ and $e^*\theta D(X)=2^X\setminus \{X\}.$ It is obvious that the set $\{b\}$ is an $e^*$-$\theta$-$D$-set but it is not $e^*$-$\theta$-open.
\end{example}

\begin{definition}
A topological space $X$ is called $e^*\text{-}\theta\text{-}D_{0}$ if for any distinct pair of points $x$ and $y$ in $X,$ there exists $e^*\text{-}\theta\text{-}D\text{-}$set $U$ of $X$ containing $x$ but not $y$ or $e^*\text{-}\theta\text{-}D\text{-}$set $V$ of $X$ containing $y$ but not $x.$ 
\end{definition}

\begin{definition}
A topological space $X$ is called $e^*\text{-}\theta\text{-}D_{1}$
if for any distinct pair of points $x$ and $y$ in $X,$ there exists $e^*\text{-}\theta\text{-}D\text{-}$set $U$ in $X$ containing $x$ but not $y$ and $e^*\text{-}\theta\text{-}D\text{-}$set $V$ of $X$ containing $y$ but not $x.$
\end{definition}
    
\begin{definition}
A topological space $X$ is called $e^*\text{-}\theta\text{-}D_{2}$
if for any distinct pair of points $x$ and $y$ in $X,$ there exist two $e^*\text{-}\theta\text{-}D\text{-}$sets $U$ and $V$ of $X$ containing $x$ and $y,$ respectively, such that $U\cap V=\emptyset.$
\end{definition}

\begin{definition}$($\cite{Ayhanozkoc}$)$
A topological space $X$ is called $e^*\theta\text{-}T_{0}$
if for any distinct pair of points $x$ and $y$ in $X,$ there exists an $e^*\text{-}\theta$-open set $U$ of $X$ containing $x$ but not $y$ and an $e^*\text{-}\theta$-open set $V$ of $X$ containing $y$ but not $x.$
\end{definition}

\begin{definition}$($\cite{Ayhanozkoc}$)$
A topological space $X$  is called $e^*\theta\text{-}T_{1}$
if for any distinct pair of points $x$ and $y$ in $X,$ there exists an $e^*\text{-}\theta\text{-}$open set $U$ of $X$ containing $x$ but not $y$ and an $e^*\text{-}\theta$-open set $V$ of $X$ containing $y$ but not $x.$ 
\end{definition}

\begin{definition} $($\cite{Ayhanozkoc}$)$ \label{7}
A topological space $X$ is called $e^*\theta\text{-}T_{2}$
if for any distinct pair of points $x$ and $y$ in $X,$ there exist two $e^*\text{-}\theta$-open sets $U$ and $V$ of $X$ containing $x$ and $y,$ respectively, such that $U\cap V = \emptyset.$
    \end{definition}
    \begin{remark}\label{tablo}
From Definitions \ref{6} to \ref{7}, we obtain the following diagram: 
    \begin{equation*}
	\begin{array}{c}
	\begin{array}{ccccccccc}
e^*\theta \text{-}T_2 & \Rightarrow & e^*\theta \text{-}T_1 & \Rightarrow & e^*\theta \text{-}T_0 \\
\Downarrow  && \Downarrow  & & \Downarrow & &&& \\ e^*\text{-}\theta \text{-}D_2 & \Rightarrow & e^*\text{-}\theta \text{-}D_1 & \Rightarrow & e^*\text{-}\theta \text{-}D_0  &  & &&\ \mbox{}\end{array}
	\end{array}%
	\end{equation*}
    \end{remark}

    \begin{theorem}$($\cite{Ayhanozkoc}$)$
Let $X$ be a topological space. If $X$ is  $e^*\theta\text{-}T_{0},$ then it is $e^*\theta\text{-}T_{2}.$
    \end{theorem}


    \begin{theorem}
Let $X$ be a topological space. If $X$ is  $e^*\text{-}\theta\text{-}D_{0},$ then it is $e^*\theta\text{-}T_{0}.$
    \end{theorem}
    \begin{proof}
It suffices to prove that every $e^*$-$\theta$-$D_0$ space is $e^*\theta$-$T_0.$ 
\\ 
Let $x,y\in X$ and $x \neq y.$ \\
      $\left.\begin{array}{rr}
     (x,y\in X) (x\neq y) \\
    X \text{ is } e^*\text{-}\theta\text{-}D_{0}
    \end{array}\right\} $ 
    $\Rightarrow (\exists A \in e^*\theta D(X,x))(y\notin A)\lor(\exists B\in e^*\theta D(X,y))(x\notin B)\\
    \begin{array}{l}\Rightarrow (\exists N,M\in e^*\theta O(X))(M \neq X)(A = M\setminus N)(x \in A)(y\notin M \lor y =M \cap N) \end{array}$ \\
   \textit{First case}:
    Let $ y\notin M.$\\
    $ \left.\begin{array}{rr}
    U:= M \\
    y\notin M
    \end{array}\right\} \Rightarrow (U \in e^*\theta O(X,x))(y \notin U)$\\
    \textit{Second case}: 
    Let $y\in M\cap N.$ \\
    $ \left.\begin{array}{rr}
    (y\in M\cap N\subseteq N)(V:= N)\\
   (A= M \setminus N)(x \in A) \Rightarrow x \notin N
    \end{array}\right\} \Rightarrow (V \in e^*\theta O(X,y))(x \notin V.)$
     \end{proof}

\begin{corollary}
  For a topological space $X$, all notions given in Remark \ref{tablo} are equivalent.     
    \end{corollary}
     
\begin{definition}
Let $X$ be a topological space, $N \subseteq X$ and $x\in X$. The set $N$ is called an $e^*\text{-}\theta\text{-}$neighbourhood of $x$ in $X$ if there exists an $e^*\text{-}\theta\text{-}$open set $U$ of $X$ such that $x \in U \subseteq N.$ The family of all $e^*\text{-}\theta\text{-}$neighbourhood of a point $x$ is denoted by $\mathcal{N}_{e^* \theta}(x).$
\end{definition}     
     
\begin{definition}
Let $X$ be a topological space and $x\in X.$ The point $x$ which has only $X$ as the $e^*\text{-}\theta\text{-}$neighbourhood is called a point common to all $e^*\text{-}\theta\text{-}$closed sets (briefly, $e^*\text{-}\theta\text{-cc}).$
\end{definition}
     
\begin{theorem}
Let $X$ be a topological space. If $X$ is $e^*\text{-}\theta\text{-}D_{1}$, then $X$ has no  $e^*\text{-}\theta\text{-cc}\text{-}point.$
\end{theorem}

\begin{proof}
Let $x,y\in X$ and $x\neq y.$
\\
$\left.\begin{array}{rr} (x,y\in X)(x\neq y) \\ X \text{ is } e^*\text{-}\theta\text{-}D_{1}\end{array}\right\} \Rightarrow (\exists A\in e^*\theta D(X,x))(\exists B \in e^*\theta D(X,y)) (x\notin B)(y\notin A) $
\\
$\begin{array}{l}
\Rightarrow (\exists U,V\in  e^*\theta O(X)) 
(U\neq X) (x\in  A= U \backslash V)
\end{array}$
 \\
$\begin{array}{l}
\Rightarrow (U\in e^*\theta O(X))(x\in U\subseteq U\neq X) 
\end{array}$
\\
$\begin{array}{l}
\Rightarrow X\neq U\in \mathcal{N}_{e^*\theta}(x)   
\end{array}$
\\
$\begin{array}{l}
\Rightarrow \mathcal{N}_{e^* \theta}(x)\neq\{X\}.
\end{array}$
\end{proof}

     \begin{definition}
A subset $A$ of a topological space $X$ is called a quasi $e^*\text{-}\theta\text{-}$closed set (briefly, $qe^*\theta$-closed) if $e^*\text{-}cl_{\theta}(A) \subseteq U $ whenever $A\subseteq U $ and $U$ is $e^*\text{-}\theta$-open in $X.$ The family of all quasi $e^*\text{-}\theta\text{-}$closed set in $X$ will be denoted by $qe^*\theta C(X).$
\end{definition}

\begin{lemma} $($\cite{Ekici2}$)$
Let $A$ be any subset of a space $X.$ Then, $x \in e^*\text{-}cl_{\theta}(A)$ if and only if $U \cap A \neq \emptyset$ for each $U\in e^*R(X,x).$
\end{lemma}

\begin{theorem} \label{4}
For a topological space $X,$ the following statements hold:\\
$(a)$ For each pair of points x and y in X, $x \in e^*\text{-}cl_{\theta}(\{y\})$ implies $y \in e^*\text{-}cl_{\theta}(\{x\}); $ 
\\
$(b)$ For each $x\in X,$ the singleton $\{x\}$ is $qe^*\theta$-closed in $X.$
\end{theorem}

\begin{proof}
$(a)$ Let $x,y \in X $ and $ y \notin e^*\text{-}cl_{\theta}(\{x\}).$ \\
$\left.\begin{array}{rr}
y \notin e^*\text{-}cl_{\theta}(\{x\}) \Rightarrow (\exists V \in e^*O(X,y))(e^* \text{-}cl(V)\cap \{x\}= \emptyset) \\
\text{Theorem } \ref{1}
\end{array}\right\} \Rightarrow
$
\\
$
\left.\begin{array}{rr}
\Rightarrow (e^*\text{-}cl(V)\in e^*R(X,y))(e^*\text{-}cl(V)\cap \{x\}=\emptyset)
\\
U:= e^*\text{-}cl(V)
\end{array}\right\} \Rightarrow
$
\\
$
\begin{array}{l}\Rightarrow (U\in e^*O(X,y))(e^*\text{-}cl(U)\cap \{x\} =\emptyset) \end{array}
$
\\
$
\begin{array}{l}\Rightarrow x\notin e^*\text{-}cl_{\theta}(\{y\}).
\end{array}
$
\\

$(b)$ Let $x\in X $, $U\in e^*\theta O(X)$ and $\{x\}\subseteq U.$ 
\\
$
\begin{array}{rr}
(x\in X)(\{x\} \subseteq U)(U\in e^*\theta O (X)) \Rightarrow U\in e^*\theta O (X,x)
\end{array}
$
\\ 
$
\left.\begin{array}{rr}
 \overset{\text{Corollary } \ref{5}}{\Rightarrow} (\exists V \in e^*R(X,x))(V \subseteq U )\Rightarrow (V\in e^*R(X,x))(V= e^*\text{-}cl(V)\subseteq U) \\ e^*R(X)\subseteq e^*\theta O(X)
\end{array}\right\}\Rightarrow
$
\\
$
\begin{array}{l}
\Rightarrow(V \in e^*O(X,x))(e^*\text{-}cl_{\theta}(\{x\})\subseteq e^*\text{-}cl_{\theta}(V)= e^*\text{-}cl(V) \subseteq U). 
\end{array}
$
\end{proof}

\begin{definition}
A space $X$ is said to be  $e^*\text{-}\theta\text{-}T_{1/2} $ if $qe^*\theta C (X)\subseteq e^* \theta C (X).$
    \end{definition}
    
    \begin{theorem}
For a topological space $X,$ the followings are equivalent: \\
$(a)$ $X$ is $e^*\text{-}\theta\text{-}T_{1/2};$ \\
$(b)$ $X$ is $e^*\theta\text{-}T_{1}.$
\end{theorem}

\begin{proof}
$(a)\Rightarrow(b):$ Let $x,y\in X$ and $x\neq y.$
\\
$\left.\begin{array}{r} x,y \in X  \overset{\text{Theorem } \ref{4}}{\Rightarrow}\{x\}, \{y\} \in qe^*\theta C (X) \\ X \text{ is } e^*\text{-}\theta\text{-}T_{1/2} \Rightarrow qe^*\theta C (X)\subseteq e^* \theta C (X) \end{array} \right\}\Rightarrow  \begin{array}{c} \\ \left. \begin{array}{r}  \{x\} ,\{y\}\in e^*\theta C (X) \\ 
x \neq y   \end{array} \right\} \Rightarrow \end{array}
$
\\
$
\left.\begin{array}{rr}
     \Rightarrow (X\setminus \{y\}\in e^*\theta O(X,x))(X\setminus \{x\}\in e^*\theta O(X,y)) \\
     (U:= X \setminus \{y\})
     (V:= X\setminus \{x\})
        \end{array}\right\} \Rightarrow$ \\
    $\begin{array}{l}\Rightarrow (U\in e^*\theta O(X,x))(V\in e^*\theta O (X,y))(y\notin U)(x \notin V).\\
    \end{array}$ 
    \\ 
    
    $(b)\Rightarrow (a):$ Let $A \in qe^*\theta C(X).$  Suppose that $A \notin e^* \theta C(X).$ We will obtain a contradiction.
    \\
    
    $
\left.\begin{array}{rr}
A \notin e^*\theta C(X) \Rightarrow A \neq e^*\text{-}cl_{\theta}(A) \Rightarrow (\exists x \in X)(x \in  e^*\text{-}cl_{\theta}(A)\setminus A)\\
X\text{ is } e^*\theta\text{-}T_{1}
\end{array}\right\} \Rightarrow $ 
        \\
        $\left.\begin{array}{rr}
\Rightarrow (\forall a\in A)(\exists V_{a}\in e^*\theta O (X,a))(x \notin V_{a}) \\ \text{Corollary } \ref{5}
        \end{array}\right\}\Rightarrow $
\\
$
\left.\begin{array}{rr}
 \Rightarrow (A\subseteq \bigcup_{a\in A} V_{a})(x \notin \bigcup_{a\in A} V_{a} \in e^* \theta O(X))\\
A\in qe^* \theta C(X)
\end{array}\right\} \Rightarrow x\notin e^*\text{-}cl_{\theta}(A)\subseteq \bigcup_{a\in A} V_{a}$
\\

This contradicts with $x \in e^*\text{-}cl_{\theta}(A).$
    \end{proof}
   
       \begin{definition}
       A function $f:X\rightarrow Y$ is said to be weakly $e^* $-irresolute  \cite{ozkocatasever} (briefly, $w.e^*.i.$) (resp. strongly $e^* $-irresolute) if for each $x\in X$ and each $V \in e^*O(Y,f(x)),$ there exists $U \in e^*O(X,x)$ such that  $f[U]\subseteq e^*\text{-}cl(V)$  $($resp. $f[e^*\text{-}cl(V)]\subseteq V).$
       \end{definition}
       \begin{remark}
 [\cite{ozkocatasever}, Theorem $3.7.(d)(e)$] A function $f:X\rightarrow Y$ is weakly $e^* \text{-}$irresolute if and only if $f^{-1}[V]$ is $e^*\text{-}\theta$-closed (resp. $e^*\text{-}\theta$-open) in $X$ for every $e^*\text{-}\theta$-closed (resp. $e^*\text{-}\theta$-open) set $V$ in $Y.$
       \end{remark}
       \begin{theorem}\label{2}
           If $f : X \rightarrow Y $ is a weakly $e^* \text{-}$irresolute surjection  and $A$ is an $e^*\text{-}\theta\text{-}$D-set in $Y,$ then the inverse image of $A$ is an $e^*\text{-}\theta\text{-}$D-set in $X$.
       \end{theorem}
       \begin{proof}
        Let $A \in e^*\theta D(Y).$\\
         $\left.\begin{array}{rr}
         A \in e^*\theta D(Y) \Rightarrow (\exists U,V \in  e^*\theta O(Y))(U \neq Y )( A = U\setminus V)\\
         f \text{ is weakly $e^*$-irresolute surjection}
         \end{array}\right\} \overset{\text{Theorem }\ref{3} }{\Rightarrow}  $
         \\
         $\begin{array}{l}\Rightarrow (f^{-1}[U], f^{-1}[V] \in e^*\theta O(X))(f^{-1}[U] \neq f^{-1}[Y]= X)(f^{-1}[A]= f^{-1}[U] \setminus f^{-1}[V])
         \end{array}$
         \\
         $\begin{array}{l}
         \Rightarrow f^{-1}[A]\in e^* \theta D (X).
         \end{array}$
       \end{proof}
       \begin{theorem}
           If $Y$ is an $e^*\text{-}\theta\text{-} D_{1}$ space and $f :X \rightarrow Y$ is a weakly $e^*$-irresolute bijection, then $X$ is $e^*\text{-}\theta\text{-} D_{1}.$ 
       \end{theorem}
    \begin{proof}
     Let $x,y \in X$ and $x\neq y . $\\
$\left.\begin{array}{r} (x,y \in X)( x\neq y) \\ f  \text{ is bijective} \end{array} \right\}\Rightarrow \begin{array}{c} \\ \left. \begin{array}{r} \!\!\!\!\!\! (f(x),f(y)\in Y)(f(x)\neq f(y))\\ 
    Y \text{ is }  e^*\text{-}\theta\text{-} D_1  \end{array} \right\} \Rightarrow \end{array}$
    \\
    $\left.\begin{array}{rr}
     \!\Rightarrow  (\exists U \in e^*\theta D (Y,f(x)))(\exists V \in e^*\theta D (Y,f(y)))(f(y)\notin U)(f(x)\notin V )\\
     \text{Theorem } \ref{2}
       \end{array}\right\} \Rightarrow $
       \\
       $\begin{array}{l}\Rightarrow  (y\notin f^{-1}[U] \in e^*\theta D (X, x)) (x \notin f^{-1}[V]\in e^*\theta D (X, y)).\end{array}$
    \end{proof}
    \begin{theorem}
        For a topological space $X,$ the followings are equivalent:\\
        $(a)$ $X$ is $e^*\text{-} \theta \text{-} D_{1};$\\
        $(b)$ For each pair of distinct points $x,y \in X,$ there exists a weakly $e^*\text{-} irresolute$
        surjection $f :X \rightarrow Y,$ where $Y$ is an $e^*\text{-} \theta \text{-} D_{1}$ space such that $f(x)\neq f(y).$
    \end{theorem}
    \begin{proof}
     $(a) \Rightarrow (b):$
     Let $x,y \in X$ and $ x \neq y.$\\
      $\left.\begin{array}{rr}
      (x,y \in X )(x \neq y)\\
      (Y:= X)(f:=\{(x,x)|x\in X\})
       \end{array}\right\} \overset{\text{Hypothesis}}{\Rightarrow} $\\
       $\begin{array}{l}\Rightarrow (f\text{ is } w.e^*.i. \text{ surjection})(Y \text{ is } e^*\text{-} \theta \text{-} D_{1})(f(x)\neq f(y)).
       \end{array}$
       \\
       
       $(b)\Rightarrow(a):$
       Let $x,y \in X$ and $x \neq y.$\\
       $\left.\begin{array}{rr}
      (x,y \in X)(x \neq y)\\
      \text{Hypothesis}
      \end{array}\right\} \Rightarrow ( \exists f\in Y^{X} \ w.e^*.i. \text{ surjection})(Y \text{ is } e^*\text{-} \theta \text{-} D_{1})(f(x)\neq f(y))$
      \\
      $\begin{array}{l}
\end{array}
$
\\
$
\left.\begin{array}{rr}
    \Rightarrow (f\in Y^{X} \ w.e^*.i. \text{ sur.})(\exists U \in e^*\theta D (Y,f(x)))(\exists V \in e^*\theta D (Y,f(y)))(f(y)\notin U)(f(x)\notin V )\\
       \text{Theorem } \ref{2}
      \end{array}\right\} \Rightarrow $\\
      $\begin{array}{l}\Rightarrow  (y\notin f^{-1}[U] \in e^*\theta D (X, x)) (x \notin f^{-1}[V] \in e^*\theta D (X, y)).
      \end{array}$
    \end{proof}
  \section[Further properties]{Further properties}

\begin{definition}
 Let $A$ be a subset of a topological space $X.$ The $e^*\text{-}\theta$-kernel of $A,$ denoted by $e^*\text{-}ker_{\theta}(A),$ is defined to be the set $\bigcap \{U| (U\in e^*\theta O(X))(A \subseteq U) \}.$  
\end{definition}

\begin{remark}
For a subset $A$ of a topological space $X$, the sets of $e^*\text{-}ker_{\theta}(A)$ and $\beta\text{-}ker_{\theta}(A)$ need not be equal to each other as shown by the following example.
\end{remark}
\begin{example}
Let $X=\{a,b,c,d\}$ and $\tau=\{\emptyset,X, \{a\}, \{a,b\},\{a,b,c\}\}.$ Then $e^*R(X)=e^*\theta O(X)=e^*O(X)=2^X$ and $\beta R(X)=\beta \theta O(X)=\{\emptyset,X\},$ $\beta O(X)=\{\emptyset,X,\{a\},\{a,b\},\{a,c\},\{a,d\},\{a,b,c\},\{a,b,d\},\{a,c,d\}\}.$ For the subset $A=\{a,b\},$   $e^*\text{-}ker_{\theta}(A)=A\neq X=\beta\text{-}ker_{\theta}(A).$
\end{example}
\begin{definition}
 A space $X$ is called slightly $e^*\text{-} \theta \text{-}R_{0 }$ space if  $\bigcap\{e^*\text{-}cl_{\theta}(\{x\})| x\in X\}=\emptyset.$
\end{definition}
\begin{remark}
A slightly $e^*$-$\theta$-$R_0$ space need not be a slightly $\beta$-$\theta$-$R_0$ space as shown by the following example.
\end{remark}

\begin{example}
Let $X=\{a,b,c,d\}$ and $\tau=\{\emptyset,X, \{a\}, \{a,b\},\{a,b,c\}\}.$ Since $\bigcap\{e^*\text{-}cl_{\theta}(\{x\})| x\in X\}=\bigcap\{e^*\text{-}cl_{\theta}(\{a\}), e^*\text{-}cl_{\theta}(\{b\}),e^*\text{-}cl_{\theta}(\{c\}),e^*\text{-}cl_{\theta}(\{d\})\}=\emptyset,$ the space $X$ is a slightly $e^*$-$\theta$-$R_0$ space. On the other hand, since $\bigcap\{\beta\text{-}cl_{\theta}(\{x\})| x\in X\}=\bigcap\{\beta\text{-}cl_{\theta}(\{a\}), \beta\text{-}cl_{\theta}(\{b\}), \beta\text{-}cl_{\theta}(\{c\}), \beta\text{-}cl_{\theta}(\{d\})\}=\bigcap \{X\}=X\neq\emptyset,$ the space $X$ is not a slightly $\beta$-$\theta$-$R_0$ space.

\end{example}
\begin{theorem} \label{kernel}
Let $A$ be a subset of a space $X.$ Then,
$e^*\text{-}ker_{\theta}(A)=\{x\in X|e^*\text{-}cl_{\theta}(\{x\})\cap A \neq \emptyset \}.$

\end{theorem}
\begin{proof}
 Let $x \notin e^*\text{-}ker_{\theta}(A).$
\\
$\begin{array}{rcl} x \notin e^*\text{-}ker_{\theta}(A) & \Rightarrow & x\notin \bigcap \{U|(U\in e^*\theta O(X))( A \subseteq U)\} \\
 & \Rightarrow & (\exists U \in e^*\theta O (X))(A\subseteq U)(x\notin U) \\
 & \Rightarrow & (\setminus U \in  e^*\theta C (X))(\{x\}\subseteq \setminus U)(\setminus U \subseteq \setminus A) \\
 & \Rightarrow & (\setminus U \in  e^*\theta C (X)) (e^*\text{-}cl_{\theta}(\{x\}) \subseteq e^*\text{-}cl_{\theta}(\setminus U)= \setminus U \subseteq \setminus A) \\
 & \Rightarrow & e^*\text{-}cl_{\theta}(\{x\})\cap A =\emptyset
 \\
 & \Rightarrow & x\notin \{x|e^*\text{-}cl_{\theta}(\{x\})\cap A =\emptyset\} \end{array}$
 \\
 
Then, we have $$\{x|e^*\text{-}cl_{\theta}(\{x\})\cap A =\emptyset\}\subseteq e^*\text{-}ker_{\theta}(A)\ldots (1)$$

Now, let $ x \notin \{x | e^*\text{-}cl_{\theta}(\{x\}) \cap A \neq \emptyset \}.
$
\\
$ \left.\begin{array}{rr}   
x \notin \{x | e^*\text{-}cl_{\theta}(\{x\}) \cap A \neq \emptyset \} \Rightarrow e^*\text{-}cl_{\theta}(\{x\}) \cap A = \emptyset\Rightarrow   A \subseteq \setminus e^*\text{-}cl_{\theta}(\{x\}) \\
U:= \setminus e^*\text{-}cl_{\theta}(\{x\})
\end{array}\right\} \Rightarrow 
$
\\
$
\begin{array}{l} \Rightarrow (U \in e^* \theta O(X))(A\subseteq U) (x\notin U)
\end{array} 
$
\\
$
\begin{array}{l}\Rightarrow x \notin \bigcap \{U|(U \in e^* \theta O (X))( A \subseteq U)\}=e^*\text{-}ker_{\theta}(A)
\end{array}
$
\\

Then, we have 
 $$e^*\text{-}ker_{\theta}(A)\subseteq \{x|e^*\text{-}cl_{\theta}(\{x\})\cap A =\emptyset\}\ldots (2)$$
 $(1),(2)\Rightarrow e^*\text{-}ker_{\theta}(A)= \{x|e^*\text{-}cl_{\theta}(\{x\})\cap A =\emptyset\}.$
\end{proof}
\begin{theorem}
Let $X$ be a topological space. Then, $X$ is slightly $e^*\text{-}\theta \text{-}R_{0}$ if and only if $e^*\text{-}ker_{\theta}(\{x\})\neq X$ for any $x\in X.$
\end{theorem}
 \begin{proof}
$(\Rightarrow):$  Suppose that there is a point $y$ in $X$ such that $e^*\text{-}ker_{\theta}(\{y\})=X.$
\\
  $e^*\text{-}ker_{\theta}(\{y\}) =\{x | e^*\text{-}cl_{\theta}(\{x\}) \cap \{y\}\neq \emptyset \} = X\Rightarrow (\forall x \in X)(y \in e^*\text{-}cl_{\theta}(\{x\}))$
  \\
  $ \left.\begin{array}{rr} 
\Rightarrow y \in \bigcap \{e^*\text{-}cl_{\theta}(\{x\})|x\in X\}\\
  \text{Hypothesis}
\end{array}\right\} \overset{\text{Theorem } \ref{kernel}}{\Rightarrow} y \in \bigcap \{e^*\text{-}cl_{\theta}(\{x\})|x\in X\}=\emptyset$ 
  \\
  
This is a contradiction.
\\

$(\Leftarrow):$ Suppose that $X$ is not slightly $e^*\text{-}\theta\text{-}R_{0}.$
\\
$
\begin{array}{rcl}
X \text{ is not slightly } e^*\text{-}\theta\text{-}R_{0} & \Rightarrow & \bigcap \{e^*\text{-}cl_{\theta}(\{x\})| x\in X\}\neq\emptyset
\\
& \Rightarrow & (\exists y \in X)(y \in \bigcap \{e^*\text{-}cl_{\theta}(\{x\})| x\in X\})
\\
& \Rightarrow & (\exists y \in X )(\forall x \in X)(y \in e^*\text{-}cl_{\theta}(\{x\})) 
\\
& \overset{\text{Theorem }\ref{kapanis}}{\Rightarrow} & (\forall x \in X)(y \in \bigcap \{V| (\{x\}\subseteq V)(V \in e^*R(X))\})
\\
&\Rightarrow& (\forall x \in X)(\forall V \in e^*R(X,y))(\{x\} \subseteq V)
\\
&\Rightarrow& (\forall V \in e^*R(X,y))(V = X)
\\
& \Rightarrow & e^*\text{-}cl_{\theta}(\{y\}) = X
\end{array}
$ 
\\
This is a contradiction.
\end{proof}
 \begin{theorem}
Let $X$ and $Y$ be two topological spaces. If $X$ is slightly $e^*\text{-}\theta\text{-}R_{0},$ then the product $X \times Y$ is slightly $e^*\text{-}\theta\text{-}R_{0}.$
 \end{theorem}
 
\begin{proof}
Let $X$ be slightly $e^*$-$\theta$-$R_0.$
\\
$\bigcap \{e^*\text{-}cl_{\theta}(\{(x,y)\}) | (x,y) \in X \times Y\} \subseteq \bigcap \{e^*\text{-}cl_{\theta}(\{x \})\times e^*\text{-}cl_{\theta}(\{y \} | (x,y) \in X \times Y\} \\
  = \bigcap \{e^*\text{-}cl_{\theta}(\{x \}) | x \in X\} \times  \bigcap \{e^*\text{-}cl_{\theta}(\{y \}) | y \in Y\} = \emptyset .$
 \end{proof}
 
 \begin{definition} \label{R-con}
  A function $f:X\rightarrow Y$ is $S$-continuous (resp. $\theta\text{-} S\text{-}e^*\text{-continuous},\\ S\text{-}e^*$\text{-continuous}) if for each $x \in X$ and each $e^*$-open subset $V$ of $Y$ containing $f(x)$, there exists an open subset $U$ of $X$ containing $x$ such that $cl(f[U]) \subseteq V$ (resp. $e^*\text{-}cl_{\theta}(f[U])\subseteq V, \  e^*\text{-}cl(f[U]) \subseteq V).$ \end{definition}

\begin{definition}\label{e-open}
  A function $f: X \rightarrow Y$ is said to be $e^*\text{-}$open  \cite{Ekici2} if $f[U]$ is  $e^*$-open in $Y$ for every open set $U$ of $X.$
\end{definition}

\begin{remark}
From Definitions \ref{R-con} and \ref{e-open}, we have the following diagram.
$$\theta \text{-}S\text{-} e^* \text{-}\text{continuous} \longrightarrow S\text{-}e^*\text{-}\text{continuous}\longleftarrow S\text{-}\text{continuous}$$
\end{remark}

A function $f$ which is $S$-$e^*$-continuous need not be $S$-continuous as shown by the following example.

\begin{example}
Let $X=\{a,b,c,d\}$ and $\tau=\{\emptyset,X,\{a\},\{c\},\{a,c\},\{c,d\},\{a,c,d\}\}.$ Define the function $f:(X,\tau) \rightarrow (X,\tau)$ by $f(x)=c.$ The function $f$ is $S$-$e^*$-continuous but it is not $S$-continuous.
\end{example}

\begin{question}
Is there any $S$-$e^*$-continuous function which is not $\theta$-$S$-$e^*$-continuous?
\end{question}
\begin{theorem}
Let $X$ and $Y$ be two topological spaces. If $f :X \rightarrow Y$ is $S\text{-}e^*\text{-continuous}$ and $e^* \text{-open},$ then $f$ is $\theta$-$S$-$e^*$-continuous.
\end{theorem}

\begin{proof}
 Let $x\in X$ and $V \in e^*O(Y,f(x))$.\\
$
\left.\begin{array}{r}(x \in X )(V \in e^*O(Y,f(x))) \\ f \text{ is} \ S \text{-} e^* \text{-} \text{continuous} \end{array} \right\}\Rightarrow \begin{array}{c} \\ \left. \begin{array}{r} \!\!\!\!\! (\exists U \in O(X,x))(e^*\text{-}cl(f[U])\subseteq V)\\ f \text{ is} \ e^*\text{-}\text{open} \end{array} \right\} \Rightarrow 
\end{array}
\\
\begin{array}{l}
\overset{\text{Theorem }\ref{1} (a)}{\Rightarrow} (\exists U \in O(X,x))(e^*\text{-}cl_{\theta}(f[U]) = e^*\text{-}cl(f[U]) \subseteq V). 
    
\end{array}$
\end{proof}  

\begin{definition}
  The graph $G(f)$ of a function $f :X \rightarrow Y$ is said to be strongly  $e^*\text{-}\theta$-closed if for each point $(x,y)\in (X \times Y)\setminus G(f),$ there exist subsets $U\in e^*O(X,x)$ and $V\in e^*\theta O(Y,y)$ such that $(e^*\text{-}cl(U) \times V) \cap G(f) = \emptyset.$
\end{definition}
\begin{lemma}
The graph  G(f) of $f :X\rightarrow Y$ is strongly $e^*\text{-}\theta$-closed in $X \times Y$ if and only if for each point $(x,y)\in(X \times Y)\setminus G(f),$ there exists $U \in e^*O(X,x)$ and $V \in e^*\theta O(Y,y)$ such that $f[e^*\text{-}cl(U)] \cap V = \emptyset.$
\end{lemma}

\begin{proof}
 Let $(x,y) \notin G(f) . $ 
 \\
 $
 \left.\begin{array}{rr}
 (x,y) \notin G(f) \\
 G(f) \text{ is strongly } e^* \text{-} \theta \text{-} \text{closed}
 \end{array}\right\} \Rightarrow 
 $
 \\ 
 $
 \begin{array}{l} \Rightarrow (\exists U \in e^*O(X,x)) (\exists V \in e^*\theta O(Y,y))((e^*\text{-}cl(U)\times V) \cap G(f) = \emptyset)
 \end{array}
 $ 
 \\
 $
 \begin{array}{l}\Rightarrow (\exists U \in e^*O(X,x)) (\exists V \in e^*\theta O(Y,y))(\forall x \in X)((x,f(x)) \notin e^*\text{-}cl(U)\times V)
 \end{array}
 $
 \\
 $
 \begin{array}{l}
 \Rightarrow (\exists U \in e^*O(X,x)) (\exists V \in e^*\theta O(Y,y))(f[e^*\text{-}cl(U)]\cap V=\emptyset).
 \end{array}
 $
 \end{proof}

\begin{definition}
A space  $X$ is called to be $e^*\text{-}T_{1}$  \cite{Ekicibul} if for each pair of distinct points in $X,$ there exist  $ e^* \text{-}$open sets $U$ and $V$ containing $x$ and $y,$ respectively, such that $y\notin U$ and $x \notin V.$ 
\end{definition}

\begin{theorem}
Let $X$ and $Y$ be two topological spaces. If  $f:X \rightarrow Y$ is $\theta\text{-}S \text{-}e^*$-continuous weak $e^*$-irresolute and $Y$ is $e^*\text{-}T_{1},$ then $G(f)$ is strongly $e^*\text{-}\theta$-closed.
\end{theorem}

\begin{proof}
Let $(x,y) \notin G(f).$
\\
$
\left.\begin{array}{rr}
(x,y) \notin G(f) \Rightarrow (y,f(x) \in Y)(y \neq f(x)) \\
Y \text{ is } e^*\text{-}T_{1}
\end{array}\right\} \Rightarrow 
$
\\
$
\left.\begin{array}{r} \Rightarrow (\exists V \in e^*O(Y, f(x))(y\notin V)) \\ f \text{ is } \theta \text{-} S \text{-} e^* \text{-} \text{continuous} \end{array} \right\}\Rightarrow \begin{array}{c} \\ \left. \begin{array}{r} \!\!\!\!\! (\exists U \in O(X,x))(y \notin e^*\text{-}cl_{\theta}(f[U]))\\ f \text{ is weak } e^*\text{-}\text{irresolute} \end{array} \right\} \Rightarrow 
\end{array}
$
\\
$
\begin{array}{l} \Rightarrow (U \in e^*O(X,x))( \setminus e^*\text{-}cl_{\theta}(f[U]) \in e^*\theta O(X,y))(e^*\text{-}cl(U)\times (\setminus e^*\text{-}cl_{\theta}(f[U])) \cap G(f) = \emptyset).
\end{array}
$
\end{proof}

\begin{theorem}
Let  $f:X \rightarrow Y$ be a weak $e^*\text{-}$irresolute function. Then, $f$ is $\theta \text{-} S \text{-} e^*$-continuous if and only if for each $x \in X $ and each $e^* \text{-} closed $ subset $F$ of $Y$ with $f(x) \notin F, $ there exists an open subset $U$ of $X$ containing $x$ and an $e^*\text{-}\theta$-open subset $V$ of $Y$ with $F \subseteq  V$ such that $f[e^*\text{-}cl(U)] \cap V = \emptyset .$
\end{theorem}

\begin{proof}
 $(\Rightarrow): $ Let $x \in X,$ $F \in e^*C(Y)$ and $f(x) \notin F.$\\
 $\left.\begin{array}{rr}
 (x \in X) (F \in e^*C(Y)) (f(x) \notin F) \Rightarrow Y \setminus F \in e^*O(Y,f(x))\\
 f \text{ is } \theta  \text{-} S  \text{-} e^* \text{-continuous} 
 \end{array}\right\} \Rightarrow  $\\
 $\left.\begin{array}{rr}
\Rightarrow  (\exists U \in O(X,x))(e^*\text{-}cl_{\theta}(f[U]) \subseteq Y \setminus F)\\
 f \text{ is weak } e^*\text{-irresolute} 
  \end{array}\right\} \Rightarrow  $
  \\
   $\left.\begin{array}{rr}
   \Rightarrow (U \in O(X,x))(f[e^*\text{-}cl(U)] \subseteq e^*\text{-}cl_{\theta}(f[U]) \subseteq Y \setminus F)\\
   V: = Y\setminus e^*\text{-}cl_{\theta}(f[U])
   \end{array}\right\} \Rightarrow  $ \\
   $\begin{array}{l}\Rightarrow (U \in O(X,x))(V \in e^*\theta O(Y))(F \subseteq V \subseteq Y \setminus f[e^*\text{-}cl(U)])
   \end{array}\\
  \begin{array}{l} \Rightarrow (U \in O(X,x))(V \in e^*\theta O(Y,F))(f[e^*\text{-}cl(U])\cap V = \emptyset .
  \end{array}$
   \\
   
   $(\Leftarrow): $ 
   Let $x \in X $ and $V \in e^*O(Y,f(x)).$\\
    $\left.\begin{array}{rr}
    (x \in X) (V \in e^*O(Y,f(x))) \Rightarrow f(x) \notin Y \setminus V \in e^*C(Y)\\
    \text{Hypothesis}
    \end{array}\right\} \Rightarrow  $\\
    $\begin{array}{l}\Rightarrow (\exists U \in O (X,x))(\exists W \in e^*\theta O(Y, Y\setminus V))(f[e^*\text{-}cl(U)]\cap W = \emptyset) 
    \end{array}\\
    \begin{array}{l}\Rightarrow (U\in O(X,x))(W \in e^*\theta O(Y, Y\setminus V))(f[U] \subseteq f[e^*\text{-}cl(U)]\subseteq Y \setminus W \subseteq V)
    \end{array}\\
    \begin{array}{l}\Rightarrow (U\in O(X,x))(e^*\text{-}cl_{\theta}(f[U])\subseteq e^*\text{-}cl_{\theta}(Y\setminus W) = Y \setminus W \subseteq V).
    \end{array}$
\end{proof}

\begin{corollary}
Let $X$ and $Y$ be two topological spaces and $f:X\rightarrow Y$ be a weak $e^*\text{-}$irresolute function. Then, f is $\theta\text{-}S\text{-}e^*\text{-}continuous $ if and only if for each $x \in X$ and each $e^*\text{-}open $ subset $V$ of $Y$ containing $f(x)$, there exists an open subset $U$ of $X$ containing $x$ such that $e^*\text{-}cl_{\theta}(f[e^*\text{-}cl(U)]) \subseteq V.$
\end{corollary}

\begin{proof}
 $(\Rightarrow): $
 Let $x \in X$ and $V\in e^*O(Y,f(x)).$ \\
$
\left.\begin{array}{r}  (x \in X) (V\in e^*O(Y,f(x)))\\ \text{Hypothesis}   \end{array} \right\}\Rightarrow \begin{array}{c} \\ \left. \begin{array}{r} \!\!\!\!\!\!  (\exists U \in O(X,x))(e^*\text{-}cl_{\theta}(f[U]) \subseteq V)   \\
 f \text{ is weak } e^*\text{-} \text{irresolute}   \end{array} \right\} \overset{\text{Theorem }\ref{3}} {\Rightarrow} \end{array}
$
\\
$
\begin{array}{l}\Rightarrow(\exists U \in O(X,x))(f[e^*\text{-}cl(U)]\subseteq e^*\text{-}cl_{\theta}(f[U])\subseteq V)
\end{array}
$
\\
$
\begin{array}{l}\Rightarrow(U \in O(X,x))(e^*\text{-}cl_{\theta}(f[e^*\text{-}cl(U)])\subseteq e^*\text{-}cl_{\theta}(e^*\text{-}cl_{\theta}(f[U])) = e^*\text{-}cl_{\theta}(f[U]) \subseteq V).
\end{array}
$
\\

  $(\Leftarrow): $ 
  Let $x \in X$ and $V\in e^*O(Y,f(x)).$\\
   $\left.\begin{array}{rr}
    (x \in X) (V\in e^*O(Y,f(x)))\\
  \text{Hypothesis}
   \end{array}\right\} \Rightarrow $\\
   $ \begin{array}{l}\Rightarrow (\exists U \in O(X,x))(e^*\text{-}cl_{\theta}(f[U])\subseteq e^*\text{-}cl_{\theta}(f[e^*\text{-}cl(U)])\subseteq V).
   \end{array}$
\end{proof}
Now, we discuss some fundamental properties of $\theta\text{-}S\text{-}e^*\text{-}$continuous functions related to composition and restriction.

\begin{theorem}
Let $f:X\to Y$ and $g:Y\to Z$ be two functions. If $f$ is continuous and $g$ is $\theta\text{-}S\text{-}e^*$-continuous, then $g\circ f:X\rightarrow Z$ is $\theta\text{-}S\text{-}e^*$-continuous. 
\end{theorem}
\begin{proof}
Let $x\in X $ and $W \in e^*O(Z,g(f(x))).$\\
$
\left.\begin{array}{r} W \in e^*O(Z,g(f(x)))\\ g\text{ is } \theta\text{-}S\text{-}e^*\text{-} \text{continuous}  \end{array} \right\}\Rightarrow \begin{array}{c} \\ \left. \begin{array}{r} \!\!\!\!\!\!  (\exists V \in O(Y,f(x))(e^*\text{-}cl_{\theta}(g[V])\subseteq W)   \\
 f \text{ is continuous}   \end{array} \right\} \Rightarrow \end{array}
$\\
$\begin{array}{l}\Rightarrow(\exists U \in O(X,x))(e^*\text{-}cl_{\theta}(g(f[U])) \subseteq e^*\text{-}cl_{\theta}(g[V])\subseteq W).
\end{array}$
\end{proof}
\begin{theorem}
 Let $f :X \rightarrow Y$ and $g :Y\rightarrow Z$ be two functions. If  $g\circ f$ is $\theta\text{-}S\text{-}e^*\text{-}continuous $ and $f$  is an open surjection, then $g$ is $\theta\text{-}S\text{-}e^*\text{-}continuous .$
\end{theorem}

\begin{proof}
Let $y\in Y $ and $W \in e^*O(Z,g(y)).$\\
$
\left.\begin{array}{r} (y\in Y ) (W \in e^*O(Z,g(y)))\\ f \text{ is  surjective}   \end{array} \right\}\Rightarrow \begin{array}{c} \\ \left. \begin{array}{r} \!\!\!\!\!\!  (\exists x\in X)(y = f(x))(W\in e^*O(Z,g(f(x))))   \\g\circ f\text{ is } \theta\text{-}S\text{-}e^*\text{-continuous} \end{array} \right\} \Rightarrow \end{array}
$\\
$\left.\begin{array}{rr}
\Rightarrow(\exists U\in O(X,x))(e^*\text{-}cl_{\theta}(g(f[U])) \subseteq W)\\
f \text{ is open}
\end{array}\right\} \Rightarrow $\\
$\left.\begin{array}{rr}\Rightarrow (f[U]\in O(Y,y))(e^*\text{-}cl_{\theta}(g[f[U]])\subseteq W)\\ V:= f[U]
\end{array}\right\}\Rightarrow$
\\
$\begin{array}{l}\Rightarrow (V\in O(Y,y))(e^*\text{-}cl_{\theta}(g[V])\subseteq W).
\end{array}$
\end{proof}

\begin{theorem}
Let $f :X \rightarrow Y$ be a function and $A\subseteq X.$  If $f$ is  $\theta \text{-}S\text{-}e^*$-continuous,  then $f|_{A}:A\rightarrow Y $ is $\theta\text{-}S\text{-}e^*\text{-}$continuous.
\end{theorem}

\begin{proof}
Let $x\in A $ and $V\in e^*O(Y,f(x)).$\\
$
\left.\begin{array}{r} (x\in A)(V\in e^*O(Y,f(x)))\\ A\subseteq X  \end{array} \right\}\Rightarrow \begin{array}{c} \\ \left. \begin{array}{r} \!\!\!\!\!\!  (x\in X)(V\in e^*O(Y,f(x)))   \\
 f \text{ is } \theta\text{-}S\text{-}e^*\text{-}\text{continuous}  \end{array} \right\} \Rightarrow \end{array}
$\\
$\left.\begin{array}{rr}
\Rightarrow(\exists W \in O(X,x))(e^*\text{-}cl_{\theta}(f[W])\subseteq V)\\
U:= W\cap A
\end{array}\right\} \Rightarrow$\\
$\begin{array}{l}\Rightarrow(U\in O(A,x))(e^*\text{-}cl_{\theta}(f|_{A}[U])=e^*\text{-}cl_{\theta}(f[W\cap A])\subseteq e^*\text{-}cl_{\theta}(f[W])\subseteq V).
\end{array}$
\end{proof}

 \section[$e^*\text{-}R_1$ space]{$e^*\text{-}R_1$ space}
\begin{definition}
A topological space $X$ is said to be $e^*\text{-}R_{1}$ if for all $x,y\in X $ with $e^*\text{-}cl(\{x\}) \neq e^*\text{-}cl(\{y\}), $ there exist disjoint $e^*$-open sets $U$ and $V$ such that $e^*\text{-}cl(\{x\}) \subseteq U $ and $e^*\text{-}cl(\{y\}) \subseteq V.$
\end{definition}
\begin{remark}
An $e^*$-$R_1$ space need not be a $\beta$-$R_1$ space as shown by the following example.
\end{remark}

\begin{example}
Let $X=\{a,b,c,d\}$ and $\tau=\{\emptyset,X,\{a\},\{c\},\{a,c\},\{c,d\},\{a,c,d\}\}.$ Then, the space $X$ is an $e^*$-$R_1$ space but it is not $\beta$-$R_1.$
\end{example}

\begin{theorem}
Let $X$ be a topological space. Then, $X$ is $e^*\text{-}R_{1}$ if and only if $e^*\text{-}cl_{\theta}(\{x\}) =e^*\text{-}cl(\{x\}) $ for all $x\in X.$
\end{theorem}
\begin{proof}
$ (\Rightarrow): $
 Let $x\in X.$\\
 $x\in X\Rightarrow e^*\text{-}cl(\{x\}) \subseteq e^*\text{-}cl_{\theta}(\{x\}) \ldots (1)$

Now, let $ y \notin e^*\text{-}cl(\{x\}).$
 \\
$
\left.\begin{array}{rr}
y \notin e^*\text{-}cl(\{x\}) \Rightarrow e^*\text{-}cl(\{x\}) \neq e^*\text{-}cl(\{y\})\\
X \text{ is } e^*\text{-}R_{1}
\end{array}\right\} \Rightarrow  
$
\\
$
\begin{array}{l}\Rightarrow (\exists U,V \in e^*O(X))(U \cap V = \emptyset)(e^*\text{-}cl(\{x\})\subseteq U)(e^*\text{-}cl(\{y\})\subseteq V)
\end{array}
$
\\
$
\begin{array}{rcl}\Rightarrow (U\in e^*O(X,x))(V\in e^*O(X,y))(e^*\text{-}cl(\{x\}) \cap e^*\text{-}cl(\{y\}) & \subseteq & e^*\text{-}cl(\{x\})\cap V \\ & \subseteq & e^*\text{-}cl(\{x\})\cap e^*\text{-}cl(V) \\ & \subseteq & e^*\text{-}cl(U) \cap e^*\text{-}cl(V) = \emptyset)
\end{array}
$
\\
$
\begin{array}{l}\Rightarrow (V \in e^*O(X,y))(\{x\}\cap e^*\text{-}cl(V) \subseteq e^*\text{-}cl(\{x\}) \cap e^*\text{-}cl(V) = \emptyset)
\end{array}
$
\\
$
\begin{array}{l}  \Rightarrow y\notin e^*\text{-}cl_{\theta}(\{x\})
\end{array}
$
\\

Then, we have $e^*\text{-}cl_{\theta}(\{x\}) \subseteq e^*\text{-}cl(\{x\}) \ldots (2)$
\\

$(1),(2) \Rightarrow e^*\text{-}cl(\{x\}) = e^*\text{-}cl_{\theta}(\{x\}).
$
\\

$(\Leftarrow):$
Let $x,y\in X$ and $e^*\text{-}cl(\{x\}) \neq e^*\text{-}cl(\{y\}).$
\\
$\begin{array}{rr}
e^*\text{-}cl(\{x\}) \neq e^*\text{-}cl(\{y\}) \Rightarrow (\exists z\in X )(z \in e^*\text{-}cl(\{x\})\setminus e^*\text{-}cl(\{y\})\vee z \in e^*\text{-}cl(\{y\})\setminus e^*\text{-}cl(\{x\})) \end{array} $
\\

\textit{First Case}: Let $z \in e^*\text{-}cl(\{x\})\setminus e^*\text{-}cl(\{y\}).$
\\
$\left.\begin{array}{rr}
z \in e^*\text{-}cl(\{x\})\setminus e^*\text{-}cl(\{y\}) \Rightarrow (z \in e^*\text{-}cl(\{x\}))(z \notin e^*\text{-}cl(\{y\})) \\
\text{Hypothesis}
\end{array}\right\} \Rightarrow $
\\
$\begin{array}{l}\Rightarrow (z \in e^*\text{-}cl(\{x\})= e^*\text{-}cl_{\theta}(\{x\})) (z \notin e^*\text{-}cl(\{y\})= e^*\text{-}cl_{\theta}(\{y\})) \end{array}\\
\begin{array}{l}\Rightarrow (\forall W \in e^*R(X,z))(W \cap \{x\} \neq \emptyset) (\exists U \in e^*R(X,z))(U \cap \{y\}= \emptyset)
\end{array}
$
\\
$
\left.\begin{array}{rr}
\Rightarrow (U \in e^*R(X,z))(\{x\} \subseteq U)(\{y\} \subseteq \setminus U)\\
V:=\setminus U
\end{array}\right\} \Rightarrow  
$
\\
$
\begin{array}{l}\Rightarrow (U,V \in e^*O(X,z))(U \cap V = \emptyset)(e^*\text{-}cl(\{x\})\subseteq U) (e^*\text{-}cl(\{y\})\subseteq V).
  \end{array}$
\\

  \textit{Second Case}: Similarly proved.
\end{proof}

\begin{theorem}
Let $X$ be a topological space. Then, $X$ is $e^*\text{-}R_{1}$ if and only if for each $e^*\text{-}open$ set $A$ and each $x\in A$, $ e^*\text{-}cl_{\theta}(\{x\}) \subseteq A.$   
\end{theorem}

\begin{proof}
$(\Rightarrow):$ 
Let $ A\in e^*O(X,x)$ and $y\notin A .$\\
$\left.\begin{array}{rr}
y\notin A\in e^*O(X,x)\\
X\text{ is } e^*\text{-}R_{1}
\end{array}\right\} \Rightarrow  x\notin e^*\text{-}cl_{\theta}(\{y\}) = e^*\text{-}cl(\{y\}) \subseteq X\setminus A  $
\\
$\begin{array}{rr}
\Rightarrow (\exists V \in e^*O(X,x))(e^*\text{-}cl(V)\cap \{y\} = \emptyset)
\end{array}  $
\\
$\left.\begin{array}{rr}
\Rightarrow (e^*\text{-}cl(V)\in e^*R(X,x))(e^*\text{-}cl(V)\cap \{y\} = \emptyset)\\
U:=\setminus e^*\text{-}cl(V)
\end{array}\right\} \Rightarrow $
\\
$\begin{array}{l}\Rightarrow (U \in e^*R(X,y)\subseteq e^*O(X,y))(e^*\text{-}cl(U)\cap \{x\} =\emptyset)
\end{array}$
\\
$\begin{array}{l} \Rightarrow y\notin e^*\text{-}cl_{\theta}(\{x\}).
\end{array}$
\\

$(\Leftarrow): $
Let $ x,y\in X $ and $y\in e^*\text{-}cl_{\theta}(\{x\}) \setminus e^*\text{-}cl(\{x\}).$
\\
$ y\in e^*\text{-}cl_{\theta}(\{x\}) \setminus e^*\text{-}cl(\{x\}) \Rightarrow(y\in e^*\text{-}cl_{\theta}(\{x\}))(y\notin e^*\text{-}cl(\{x\}))$
\\
$\left.\begin{array}{rr}
\Rightarrow (y\in e^*\text{-}cl_{\theta}(\{x\}))(\exists A\in e^*O(X,y))(A\cap \{x\} = \emptyset)\\
\text{Hypothesis}
\end{array}\right\} \Rightarrow  $\\
$\begin{array}{l}\Rightarrow (y\in e^*\text{-}cl_{\theta}(\{x\}))(e^*\text{-}cl_{\theta}(\{y\})\cap \{x\} =\emptyset)
\end{array}\\
\begin{array}{l}\Rightarrow (y\in e^*\text{-}cl_{\theta}(\{x\})(x\notin e^*\text{-}cl_{\theta}(\{y\}))
\end{array}\\
\begin{array}{l} \overset{\text{Theorem } \ref{4}}\Rightarrow(y\in e^*\text{-}cl_{\theta}(\{x\})(y\notin e^*\text{-}cl_{\theta}(\{x\}))
\end{array}\\
\begin{array}{l}\Rightarrow y\in e^*\text{-}cl_{\theta}(\{x\})\setminus e^*\text{-}cl_{\theta}(\{x\}) = \emptyset.
\end{array}$ 
\\

This is a contradiction.
\end{proof}

\begin{theorem}
Let $X$ and $Y$ be two topological spaces. If $f :X \rightarrow Y$  is a $\theta\text{-}S\text{-}e^*\text{-}continuous $ surjection, then  $Y$ is an $e^*\text{-}R_{1}$ space.
\end{theorem}

\begin{proof}
Let $V\in e^*O(Y,y).$
\\
$
\left.\begin{array}{r}  V\in e^*O(Y,y)\\ f  \text{ is surjective} \end{array} \right\}\Rightarrow \begin{array}{c} \\ \left. \begin{array}{r} \!\!\!\!\!\!  (\exists x \in X)(y= f(x))(V\in e^*O(Y,f(x)))\\
  f \text{ is }\theta\text{-}S\text{-}e^*\text{-}\text{continuous}    \end{array} \right\} \Rightarrow \end{array}
$
\\
$
\begin{array}{l}\Rightarrow (\exists U \in O(X,x))(e^*\text{-}cl_{\theta}(\{y\}) \subseteq e^*\text{-}cl_{\theta}(f[U])\subseteq V).
\end{array}
$
\end{proof}

\section{Acknowledgements}
This study has been supported by the Scientific Research Project Fund of Muğla Sıtkı Koçman University under the project number 23/153/02/1.



\begin{thebibliography}{99}
{\smaller


\bibitem{Ayhanozkoc} B.S. Ayhan and M. Özkoç, On almost contra $e^*\theta \text{-}$continuous functions, Jordan J. Math. Stat. 11 (4), (2018), 383-408.
\bibitem{Caldas} M. Caldas, S. Jafari and T. Noiri, On the class of semipre-$\theta$-open sets in topological spaces, Selected papers of the 2014 International Conference
on Topology and its Applications, (2015), 49-59.

\bibitem{Ekicibul} E. Ekici, Some generalizations of almost contra-super-continuity, Filomat. 21 (2) (2007), 31-44.

\bibitem{Ekici1} E. Ekici, New forms of contra continuity, Carpathian J. Math. 24 (1) (2008b), 37-45.

\bibitem{Ekici2} E. Ekici, \textit{On $e^*$-open sets and $(\mathcal{D},\mathcal{S})^*$-sets}, Math. Morav. 13 (2009), 29-36.

\bibitem{Ozkoc1}  A.M. Farhan and X.S. Yang, New type of strongly continuous functions in topological spaces via  $ \delta \text{-} \beta \text{-}$open sets, Eur. J. Pure Appl. Math. 8 (2) (2015),
 185-200.


\bibitem{Jafari} S. Jafari and T. Noiri, On strongly $\theta$-semi-continuous functions, Indian J. Pure Appl. Math. 29 (1998), 1195-1121.





\bibitem{Noiri} T. Noiri, Strongly $\theta$-precontinuous functions, Acta  Math. Hungar. 90 (2001), 307-316.
 
 \bibitem{Popa} T. Noiri and V. Popa, Strongly $\theta$-$\beta$-continuous functions, J. Pure Math. 19 (2002), 31-39.

 

 \bibitem{ozkocatasever} M. Özkoç and K.S. Atasever, On some forms of $e
^*\text{-}$irresoluteness, J. Linear Topol. Algeb. 08(01), (2019), 25-39.

 


\bibitem{Stone} M.H. Stone, Applications of the theory of Boolean rings to general topology, Trans. Amer. Math. Soc. 41 (1937), 375-381. 




\bibitem{Velicko} N.V. Velicko, $H$-closed topological spaces, Amer. Math. Soc. Transl. 78 (1968), 103-118.





\label{pgCAbdio}
}
\end{thebibliography}
\end{document}